\documentclass[a4paper]{article}
\usepackage{RR}
\usepackage{graphics}
\usepackage{amsmath,amsthm,amsfonts,amssymb}
\usepackage[all,dvips]{xy}
\usepackage{graphicx}
\usepackage[latin1]{inputenc}
\usepackage{mathrsfs}	
\usepackage[mathcal]{euscript}
\RRdate{Mars 2009}

\RRauthor{
Etienne Farcot\thanks{Virtual Plants INRIA, UMR DAP, CIRAD, TA A-96/02, 34398 Montpellier Cedex 5, France, farcot@sophia.inria.fr}%
  \and
Jean-Luc Gouz\'e\thanks{COMORE INRIA, Unit\'{e} de recherche Sophia Antipolis Méditerranée , 2004 route des Lucioles, BP 93, 06902 Sophia Antipolis, France, gouze@sophia.inria.fr}
}
\authorhead{Farcot \& Gouz\'e}
\RRtitle{Cycles limites dans les modèles affines par morceaux de r\'eseaux génétiques incluant plusieurs boucles d'interaction}
\RRetitle{Limit cycles in piecewise-affine gene network models with multiple interaction loops}
\titlehead{Limit cycles in gene network models}
\RRnote{This work was partly supported by the INRIA initiative "Colage".}
\RRresume{Ce rapport porte sur des équations différentielles affines par morceaux utilisées comme modèles de réseaux génétiques. Les taux de dégradation sont arbitraires, et nous faisons une hypothèse locale d'alignement de points focaux successifs. Le graphe d'interaction du système peut être assez complexe (multiples boucles d'interaction de tous signes, seuils multiples...). Le résultat principal est un théorème d'alternative montrant que si une suite de régions est périodiquement visitée par les trajectoires du système, alors sous nos hypothèses, soit il existe une unique solution périodique stable, soit l'origine est attractante pour toutes les trajectoires dans la suite de régions. Ce résultat étend nos travaux précédents sur une boucle de retroaction négative. Plusieurs exemples et simulations sont présentés pour illustrer la méthode.}
\RRabstract{In this paper we consider piecewise affine differential equations modeling gene networks. We work with arbitrary decay rates, and under a local hypothesis expressed as an alignment condition of successive focal points. The interaction graph of the system may be rather complex (multiple intricate loops of any sign, multiple thresholds...). Our main result is an alternative theorem showing that, if a sequence of region is periodically visited by trajectories, then under our hypotheses, there exists either a unique stable periodic solution, or the origin attracts all trajectories in this sequence of regions. This result extends greatly our previous work on a single negative feedback loop. We give several examples and simulations illustrating different cases.}
\RRmotcle{syst\`emes dynamiques linéaires par morceaux ; trajectoires periodiques ; applications monotones, concaves ; graphe d'interaction ; mod\`eles de r\'eseaux g\'en\'etiques}
\RRkeyword{piecewise linear dynamical systems ; periodic trajectories ; monotone, concave maps ; interaction graph ; genetic network models}
\RRprojets{Virtual Plants et Comore}
\RRtheme{\THBio} 
\RCSophia 

\newcommand{\ee}{{\bf e}}
\newcommand{\dd}{{\sf d}}

\newcommand{\R}{\mathbf{R}}

\newcommand{\kk}{\kappa }

\newcommand{\D}{\mathscr{D}}
\newcommand{\Dd}{{\cal D}}
\newcommand{\A}{\mathcal{A}}
\newcommand{\T}{T}
\newcommand{\TT}{\mathbf{T}}

\newcommand{\zo}{\mathcal{Z}}

\newcommand{\ie}{i.e. }
\newcommand{\intr}[1]{{\rm int}\left(#1\right)} 
\newcommand{\cl}[1]{\overline{#1}} 
\newcommand{\sgn}{{\rm sign}\,}

\newcommand{\NN}[1]{\{1\cdots #1\}}
\newcommand{\Nm}[1]{\{0\cdots #1-1\}}

\newtheorem{assum}{Assumption}
\newtheorem{prop}{Proposition}

\newtheorem{theo}{Theorem}

\newtheorem{rem}{Remark}

\begin{document}
\RRNo{6875}
\makeRR   



\section{Preliminaries}\label{sec-intro} 
It has been discovered in the 1960's that some proteins can regulate (i.e. activate or inhibit) the expression of genes in a living organism. Since proteins are the product of gene expression, feedback appears to be at the core of this process. Moreover, the unprecedented developments of cell biology in the last decades has led to a view where gene regulation involves huge numbers of elements (genes, mRNAS, proteins ...), interacting in a nonlinear way. It is thus clear today that mathematical models and tools are required to analyse these complex systems. Several classes of models have been proposed to describe gene regulation. Although stochastic models undoubtedly have a major role to play, we will focus in this note on deterministic models. They can be divided in two classes : models based on differential equations \cite{murray}, and discrete models based on a representation by a finite number of states \cite{kauffman,thodari}. Both classes have their complementary advantages, the most obvious being relative to the accuracy vs. tractability dilemma. First proposed by L. Glass \cite{glasstopkin}, piecewise linear (in fact piecewise affine) models appear to be an efficient intermediate between the two previous classes. Thus, the sometimes called 'Glass systems' have been both used to model real gene networks \cite{bacillus,ropers}, and studied mathematically. The present paper is more related to this second type of work.\\
The analysis of periodic solutions in piecewise affine gene network models is well characterised in the special case where all decay rates are supposed equal \cite{glasspastern2,edwards,farcot,periodsol}. With non uniform decay rates however, very few results are known, and the techniques developed for homogeneous decay rates cannot be generalised. This lack of results is unfortunate, since real gene networks are known to involve very distinct degradation rates, both among mRNAs and proteins. Moreover, the question is also relevant from a more mathematical viewpoint, since systems with uniform decay rates display appreciably simpler dynamics than systems with distinct decays. Actually, in the uniform decay rate setting, trajectories are locally straight lines, whereas distinct decays lead to pieces of exponential curves. As an illustration of this, a 3-dimensional example of chaotic behaviour has been provided \cite{chaosglass3d}, whereas chaos in a uniform decay rates setting has been proved to require at least 4 dimensions \cite{periodsol}.\\
In a previous paper~\cite{negloop,actabio}, we have successfully used tools from the theory of monotone systems and operators to tackle this problem. This has lead us to prove the existence of a unique stable limit cycle in an invariant region of state space that is periodically crossed by flow lines. This result holds only for a particular class of systems, namely those having a single negative loop as an interaction graph.\\
However, the essential property of these systems that was mathematically used in the proof was a geometric condition. It was expressed as an alignment condition satisfied by pairs of successive focal points (see below for a precise definition), which was shown to be always true in the case when the interaction graph is a negative loop involving all variables of the system. In this paper, we generalise our previous results to a setting where the main assumption is only this alignment of successive focal points, regardless of the interaction structure. It is possible to show that this alignment assumption necessarily holds when each node in the interaction graph has outgoing degree 1. But the notion of interaction graph is essentially local, and the previous condition only has to hold at each point in state space (or in a region of interest in state space). This includes of course feedback loops, but also many other interaction structures, including systems whose global interaction graph has a complex form, as some examples will illustrate.\\
 Since all these results require some technicalities, which have been partly addressed in \cite{negloop,actabio}, the proofs provided in this paper are not entirely self-contained, but the missing parts appear in the cited references.\\
The following section introduces the class of piecewise affine models of gene networks. Then in a next section, a general fixed point theorem for monotone systems is recalled, and applied to a class of piecewise affine systems. It is illustrated on several examples in a final section.

\section{Piecewise affine models}\label{sec-pwaff}
\subsection{Definitions and notations}\label{sec-defs}
In this section we recall basic facts about piecewise affine models \cite{glasstopkin,edwards,farcot,inria}. The general form of these models can be written as:
\begin{equation}\label{eq-genenet}
\frac{dx}{dt} = \kk(x) - \Gamma x
\end{equation}
The variables $(x_1\dots x_n)$ represent concentrations in proteins or mRNA produced from $n$
interacting genes. By abuse of language, the $n$ elements of the system will be called {\it genes} in the sequel. 
Since gene transcriptional regulation is often considered to follow a steep
sigmoid law, it has been suggested that idealised, discontinuous switches may be used instead
to model these complex systems~\cite{glasstopkin}. Accordingly, let us denote ${\sf s}^+(\cdot\,,\theta):\R\to \{0,1\}$ 
the increasing step function: 
\[
\left\{\begin{array}{lcl}
{\sf s}^+(x,\theta) & = & 0\quad\text{if } x<\theta,\\
{\sf s}^+(x,\theta) & = & 1\quad\text{if } x>\theta,
\end{array}\right.
\]
which represent an effect of activation. Also, ${\sf s}^-(x,\theta)=1-{\sf s}^+(x,\theta)$, is its decreasing version, and represents inhibition. Unless further precision are given, we leave this function undefined at its threshold value $\theta$.\\
The map $\kk:\R_+^n\to\R^n_+$ depends on $x$ only via step functions of the form ${\sf s}^\pm(x_i,\theta_{i})$, and is thus piecewise constant. $\Gamma\in\R_+^{n\times n}$ is a diagonal matrix whose diagonal entries $\Gamma_{ii}=\gamma_i$, are degradation rates of variables in the system.\\ 
As a concentration, each variable $x_i$ is nonnegative and bounded. When $x_i$ reaches a threshold value, this has an instantaneous effect on the system: for some $j$, the value of $\kk_j$ changes. In other words, a step function ${\sf s}^\pm(x_i,\theta_i)$ appears in the expression of the function $\kk_j$. For each $i\in\NN{n}$ let us define a finite set of threshold values:
\begin{equation}\label{eq-thresh}
\Theta_i=\{\theta_{i}^0,\dots,\theta_{i}^{q_i}\},
\end{equation}
where the thresholds are ordered: $\theta_{i}^0=0<\theta_{i}^1<\dots<\theta_{i}^{q_i-1}<\theta_i^{q_i}$.  The extreme values $0$ and $\theta_i^{q_i}$ are not thresholds but bounds on $x_i$'s value, but it will convenient to denote them as elements of $\Theta_i$.\\ 
Now, each axis of the state space can be partitioned into open segments between thresholds. Since the extreme values will not be crossed by the flow (see later), the first and last segments include one of their endpoints~:
\begin{equation}\label{eq-Di}
{\Dd}_i\in \left\{[\theta_i^0,\,\theta_i^1),\;\{\theta_i^1\},\;(\theta_i^1,\,\theta_i^{2}),\;\{\theta_i^2\},\cdots,\{\theta_i^{q_i-1}\},\; (\theta_i^{q_i-1},\,\theta_i^{q_i}]\right\}
\end{equation}
Each product ${\Dd}=\prod_{i=1}^n{\Dd}_i$ defines a rectangular {\it domain}, whose dimension
is the number of ${\Dd}_i$ that are not singletons. When $\dim{\Dd}=n$, one usually says that
it is a {\em regulatory domain}, or {\em regular domain}, and those domains with lower
dimension are called {\em switching domains}~\cite{inria}. We use the notation $\D$ to represent the set
of all domains of the form $\Dd$ above, and $\D_r$ (resp. $\D_s$) for the set of all regulatory (resp. switching) domains. 
Since $\D_r$ is composed of finitely many domains, it will be convenient to identify it with
\begin{equation}\label{eq-A}
\A = \prod_{i=1}^n\Nm{q_i}.
\end{equation}
In the rest of the paper, this will lead us to use formulations like 'in a domain $a$'. This identification can be described by a map $\dd:\D_r\to\A$, $\dd\left(\prod_i (\theta_i^{a_i-1},\,\theta_i^{a_i})\right)=(a_1-1\dots a_n-1)$.\\ 
The dynamics on regular domains, or {\it regular dynamics}, is easy to describe, see next section. On sets of $\D_s$ on the other hand, the flow is in general not uniquely defined. This can be circumvented by using set-valued solutions and the theory of Filippov~\cite{casey,gouzesari}. However, in this paper we will not need this theory, thanks to some mild assumptions explained in the next section.

\subsection{Regular dynamics}\label{sec-regdyn}
On any regular domain $\Dd$, the production rate $\kk$ is constant, and thus equation~(\ref{eq-genenet}) is affine. Its solution is explicitly known, for each coordinate $i$~:
\begin{equation}\label{eq-flow}
\varphi_i(x,t)=x_i(t) = \frac{\kk_i}{\gamma_i} - e^{-\gamma_i t}\left(x_i -
\frac{\kk_i}{\gamma_i}\right),
\end{equation}
and is valid for all $t\in \R_+$ such that $x(t)\in{\Dd}$. One sees above that the point
\[
\phi({\Dd}) =
\left(\phi_1\cdots\phi_n\right)=\left(\frac{\kk_1}{\gamma_1}\cdots\frac{\kk_n}{\gamma_n}\right)
\]
It is an attractive equilibrium of the flow~(\ref{eq-flow}). Hence, if it lies inside $\Dd$, it is an
asymptotically stable steady state of system~(\ref{eq-genenet}). Otherwise, the flow will reach the boundary
$\partial{\Dd}$ in finite time, unless $\phi({\Dd})$ lies exactly on the  boundary of $\Dd$. However this
situation is clearly not generic, and thus one will always suppose in the following that:
\[
\forall\Dd\in\D_r,\quad \phi({\Dd})\in\D_r.
\] 
When the flow reaches $\partial{\Dd}$, the value of $\kk$ (and thus, of $\phi$) changes, and the flow changes its direction. The point $\phi({\Dd})$ is often called \textit{focal point} of the domain $\Dd$. Note that if $\dd(\Dd)=a$, we will often denote it $\phi(a)$, or $\phi^a$.\\
It follows that the continuous trajectories are entirely characterised by their successive intersections with the boundaries of regular domains, and that this sequence depends essentially on the position of focal points. However, the definition of trajectories on the boundary of regular domains requires further explanations. Let us describe the case of singular domains of dimension exactly $n-1$.  Let $W$ be a $n-1$ dimensional domain, intersecting the boundaries of two regular domains $\Dd$ and $\Dd'$. If in these two domains flow lines both point towards, or away from $W$ (i.e. the flow coordinate in the direction normal to $W$ has different signs in $\Dd$ and $\Dd'$) the latter is called respectively {\it black wall} or {\it white wall}. In both cases,  the Filippov theory~\cite{gouzesari}, or some other technique, is required. Otherwise, \ie when flow lines both cross $W$ in the same direction, one usually call it a {\it transparent wall}, and trajectories on $W$ can be defined by continuity, from $\Dd$ and $\Dd'$.\\
In the following, one will only deal with transparent walls. A simple criterion to ensure that all walls are transparent is the absence of auto-regulation, in the sense that no production term $\kk_i$ depends on $x_i$. Then, the only regions that are excluded are the singular domains of co-dimension $2$ or more, which form a rare set in state space (and even in $\D_s$).\\

Given the flow (\ref{eq-flow}) in a box ${\Dd}$ (of image $a$ in $\A$), it is easy to compute the time and position at which it intersects the boundary of ${\Dd}$, if ever. The position of the focal point with respect to thresholds determines entirely which walls can be reached: $\{x\,|\,x_i=\theta_{i}^{a_i-1}\}$ (resp. $\{x\,|\,x_i=\theta_{i}^{a_i}\}$) can be crossed if and only if $\phi_i < \theta_{i}^{a_i-1}$ (resp. $\phi_i>\theta_{i}^{a_i}$). Then, let us denote $I_{out}^+(a) = \{i\in\NN{n} | \,\phi_i >\theta_{i}^{a_i}\}$, and $I_{out}^-(a) = \{i\in\NN{n} | \,\phi_i <\theta_{i}^{a_i-1}\}$. Similarly, $I_{out}(a)=I_{out}^+(a)\cup I_{out}^- (a)$ is the set of escaping directions of $\Dd$. Also, the following pairs of functions will be convenient notations: $\theta_i^\pm:\A\to\Theta_i$, $\theta_i^-(a)=\theta_{i}^{a_i-1}$ 	and  $\theta_i^+(a)=\theta_{i}^{a_i}$.\\ 
When it is unambiguous, we will omit the dependence on $a$ in the sequel.\\ 
Now, in each direction $i\in I_{out}$ the time at which $x(t)$ encounters the corresponding hyperplane, for $x\in {\Dd}_a$, can easily be shown to be:
\begin{equation}\label{eq-taui}
\tau_i(x)=\frac{-1}{\gamma_i}\ln\left(\min\left\{\frac{\phi_i -\theta_{i}^-(a)}{\phi_i -x_i},
\frac{\phi_i -\theta_{i}^+(a)}{\phi_i -x_i}\right\}\right). 
\end{equation}
Then, 
$\tau(x)=\min_{i\in I_{out}}\tau_i(x)$,
is the time at which the boundary is crossed by the trajectory originated at $x$. Then, the escaping point of ${\Dd}$ from initial condition $x$ takes the form $\varphi(x,\tau(x))$. Since this will be repeated along trajectories, $x$ will generally lie on the boundary of the current box, except for the initial condition, which may however be supposed to lie on a wall without loss of generality. In this way, one defines a \textit{ transition map} ${\T}^a: \partial \Dd\rightarrow \partial \Dd$:
\begin{equation}\label{eq-maptrans}
\begin{array}{lcl}
{\T}^ax & = & \varphi\left(x,\tau(x)\right)\\
          & = & \phi + \alpha(x) (x-\phi).
\end{array}
\end{equation}
where $\alpha(x) = \exp(-\tau(x)\Gamma)$. The latter depends on $a$, as seen
from~(\ref{eq-taui}).\\ 
The map above is defined locally, in a domain $a$. However, under our assumption that all considered walls are transparent, any wall can be unequivocally considered as escaping in one of the two regular domains in bounds, and incoming in the other. Hence, on any point of the interior of a transparent wall, there is no ambiguity on which $a$ to chose in expression~(\ref{eq-maptrans}). In other words, there is a well defined global transition map on the union of transparent walls. Let us denote this map $\T$. \\
Now, the initial system~(\ref{eq-genenet}) may been reduced to a discrete time dynamical system $({\rm Dom}_{\,\T},\T)$, where ${\rm Dom}_{\,\T}$ is the subset of  $n-1$ domains of $\D_s$ where all iterates of $\T$ are defined. From the previous discussions, it appears that ${\rm Dom}_{\,\T}$ is the union of all (transparent) walls, minus the union of all finite-time preimages (i.e. finite number of backward iterates of $\T$) of $n-2$ dimensional (or less) singular domain. The topology of this domain is not trivial in general, and is described with more detail in~\cite{farcot}.\\ 
Now, we will focus on situations where there is a wall $W$, and a sequence $a^1\dots a^\ell$ of regular domains such that $({\T}^{a^\ell}\circ{\T}^{a^{\ell-1}}\cdots\circ{\T}^{a^1}) (W)\cap W\ne \varnothing$.  Then, fixed points of such an iterate are equivalent to periodic trajectories in the original system~(\ref{eq-genenet}).\\


\section{An alternative theorem for convergence}\label{sec-theory} 
Let us state some general notations. We denote $x<y$ and $x\leqslant y$ if these inequalities hold for each coordinate (resp. entry) of vectors (resp. matrices) $x$ and $y$. We call this order the partial order. Then, we denote $x\lneq y$ if $x\leqslant y$ and $x \neq y$. For $x\leqslant y$, $[x,y]=\{z\,|\, x\leqslant z\leqslant y\}$, and $(x,y) = \{ z \,|\, x <z< y\}$. For any set $A$, by $\intr{A}$ we denote the interior of $A$, and by $\cl{A}$ its closure.\\

\subsection{Theorems for monotone systems}
Let us state first the main theorem we want to apply. It is a fixed point theorem for monotone and concave operators, with respect to the partial order. Many variants of this theorem have been proposed since early works, more than 50 years ago. The form we use is due to Smith \cite{smith}. In words, it states that a monotone and concave map on a compact domain of $\R_+^n$ may have either the origin as a unique fixed point, or a unique positive (for the partial order) and attracting fixed point. The second case happens when either the origin is an unstable fixed point, or when it is not a fixed point at all. Let us now state this theorem:
\begin{theo}\label{thm-smith} Let $p\in{\R}^n_+$, $p>0$, and $T:[0,p] \to [0,p]$
continuous, $C^1$ in $(0,p)$.\\ 
Suppose $DT(0)=\displaystyle\lim_{\substack{x\to 0\\ x>0}} DT(x)$ exists.
Assume: 
\begin{itemize} 
	\item[] $\qquad\qquad\qquad$(M) $\qquad DT(x) > 0$ if $\;x>0$, $x<p$. 
	\item[] $\qquad\qquad\qquad$(C) $\qquad DT(y) \lneq DT(x)$ if $\;0 < x < y < p$. 
\end{itemize} 
Assume also $Tp < p$.\\ 
If $T0 = 0$, let $\lambda =\rho(DT(0))$, the spectral radius of $DT(0)$. Then,  
\begin{itemize} 
\item[]$\lambda \leqslant 1\implies \forall x\in[0,p]$, $T^n x\to 0$ when $n\to\infty$. 
\item[]$\lambda > 1\implies$ There exists a unique nonzero fixed point $q=Tq$. Moreover, 
$q\in (0,p)$ \\
\phantom{$\lambda > 1\implies$ }and for every $x\in[0,p]\setminus\{0\}$, $T^n x\to q$ as $n\to\infty$. 
\end{itemize} 
If $T0\gneq 0$, then $T$ has a unique fixed point $q\in[0,p]$. Moreover, $q\in (0,p)$ and $T^n x\to q$ as $n\to\infty$ for every $x\in[0,p]$. 
\end{theo}

One may remark now that in the case when $T$ has second-order derivatives, the concavity condition $(C)$
admits a simple sufficient condition.  Let the map $T_i:[0,p]\to\R_+$ denote the $i$th coordinate function
of $\, T:[0,p]\to[0,p]$.
\begin{prop}\label{prop-D2C}
Suppose that for all $\;i,j,k\in\NN{n}$, and for all $0< x <p$,
\[
\frac{\partial^2 T_i}{\partial x_k\partial x_j}(x) \leqslant 0,
\]
and for all $i,j$ there exists a $k$ such that the inequality is strict.\\ 
Then $T$ satisfies condition $(C)$ of theorem~\ref{thm-smith}.
\end{prop}
Actually, it is clear that under this condition each term $\frac{\partial T_i}{\partial x_j}$ of $DT$ is a
decreasing function of each coordinate $x_k$. It is moreover strictly decreasing in at least one of these
coordinates, and $(C)$ thus follows. Observe by the way that the notion of concavity (w.r.t a partial order) we
deal with here is weakened by the fact that it concerns only ordered pairs $(x,y)$ of variables.\\

\subsection{Preliminary results for piecewise-affine models}
Let ${\cal C}=\{a^0,a^1\cdots a^{\ell-1}\}$ denote a sequence of regular domains which is periodically visited by the flow. Thus, we will consider $a^\ell=a^0$, and more generally the upperscript $i$ in $a^i$ shall be understood modulo $\ell$, unless explicitly mentioned. Also, we will denote walls in $\cal C$ as follows: $W^i=\cl{a^i}\cap\cl{a^{i+1}}$. Then for $i\in\{0\cdots\ell\}$ we use the notation $\phi^i=\phi(a^i)$, and $\theta^{(i)}_{s_i}$ is the $i$th threshold, i.e. $W^i\subset\{x\in\R^n\,|\, x_{s_i}=\theta^{(i)}_{s_i}\}$\footnote{Remark that walls have be defined as closed sets. The transition maps are originally only defined on the interior of these regions. However, it is not difficult to show that they can be extended at the boundary of these set in the present context\index{}.}.\\
We define local and global transition maps as follows:
\begin{equation}\label{eq-T}
\xymatrix@+.7pc{
W^0 \ar@<-.5pc>_-{\TT=T^{(\ell)}\circ T^{(\ell-1)}\cdots\circ T^{(1)}}[rrr] \ar^-{T^{(1)}}[r] & W^{(1)} \ar^-{T^{(2)}}[r] & \;\cdots\; \ar[r]^-{T^{(\ell)}} & W^{\ell} =W^0 \\
}
\end{equation}

We denote $a^{i+1}-a^i = \varepsilon_i \ee_{s_i}$, i.e. $s_i$ is the exiting direction of box $a^i$, and $\varepsilon_i\in\{-1,+1\}$ indicates whether trajectories leave this box increasingly or decreasingly in direction $\ee_{s_i}$.\\

Now, let us give a more detailed expression of a local transition map $T^{(i)}$, obtained after straightforward calculations:
\begin{equation}\label{eq-loctransmap}
{T}^{(i)}(x) = 
\Big(\, \phi^i_j + (x_j - \phi^i_j)\alpha^{(i)}_j(x) \,\Big)_{\substack{j=1\dots n\\ j\ne s_i}}
\end{equation}
where $\displaystyle \alpha^{(i)}_j(x) =\alpha^{(i)}_j(x_{s_i}) = \left(\frac{\phi^{i}_{s_i} - \theta^{(i)}_{s_i}}{\phi^{i}_{s_i} - x_{s_i}}\right)
^{\frac{\gamma_j}{\gamma_{s_i}}}$. The $j$th coordinate map of ${T}^{(i)}$ is denoted ${T}^{(i)}_j$.\\

Furthermore, we make the
\begin{assum}\label{assum} For each $i\in\Nm{\ell}$, for all $j\in\NN{n}\setminus\{s_{i+1}\}$, $\phi_j^{i+1}=\phi_j^i$. Pairs of successive focal points satisfying this condition will be said to be \emph{aligned}.
\end{assum}

One may observe that there is a first dichotomy in theorem~\ref{thm-smith} : depending on whether $0$ is a fixed point or not, one has to check a condition on the spectral radius of $DT(0)$ or not. Here, $0$ will be some corner of the wall we chose as a Poincar\'e section, hence a point with all its coordinates equal to some threshold.  As might be seen by some hand drawing, it appears that the case $T0\gneq 0$ happens when several parallel threshold hyperplanes are crossed in some direction, while $0$ is a fixed point if there is a single threshold in every direction (which is translated to $0$). Let us state this fact more precisely.

\begin{prop}\label{prop-parallel}
Suppose that there are two distinct crossed thresholds in at least one direction. Assuming without loss of generality that one of these defines the wall $W^0$, i.e. $s_0$ is a direction with two crossed thresholds, one may write this condition as follows
$$
\exists i\in\NN{n},\,\exists k\in\NN{\ell},\qquad i=s_0=s_k\quad\text{and}\quad\theta^{(0)}_{s_0}\ne \theta^{(k)}_{s_k}.
$$
Then, $T^{(k)}\circ T^{(k-1)}\cdots\circ T^{(1)} \left(W^0\right) \subset \intr{W^k}$, and consequently
$\displaystyle
\TT \left(W^0\right)\subset \intr{W^0}.
$
\end{prop}
To ease the reading, we have postponed the proof of this proposition to Annex \ref{anx-proofs}.\\

\begin{rem}The previous proposition holds for any disposition of the focal points, but requires the assumption that each box in the sequence admits a unique escaping direction.
\end{rem}
Another important result that we need is the existence of a region on which $\TT$ is monotone. If so, concavity will follow. Under the alignment condition on focal points, there will be a unique region on which this holds after a finite number of iterates. We shall prove this fact.\\
First of all, let us compute the first order derivatives of a transition map $T^{(i)}$ at a point $x$:
\begin{equation}\label{eq-dTidxj}
\frac{\partial\T_k^{(i)}}{\partial x_j} (x) =\left\{
\begin{array}{ll}
\displaystyle \alpha_k^{(i)}(x) & \text{if }\;  k=j\\[2mm]
\displaystyle-\frac{\gamma_k}{\gamma_{s_i}}\frac{\phi_k^i - x_k}{\phi_{s_i}^i - x_{s_i}}
\alpha_k^{(i)}(x) & \text{if }\;  j=s_i \\[2mm]
0 & \text{otherwise}.
\end{array}\right.
\end{equation}
Where $ j\in\NN{n}\setminus\{s_{i-1}\}$ and $k\in\NN{n}\setminus\{s_i\}$.\\
One deduces that diagonal terms of the Jacobian are positive, on the column $s_i$ one has
$$
\sgn\left(\frac{\partial\T_k^{(i)}}{\partial x_{s_i}}(x) \right) =- \sgn\left(\phi_k^i - x_k\right) \sgn\left(\phi_{s_i}^i - x_{s_i}\right) ,\quad k\in\NN{n}\setminus\{s_i\} ,
$$
and the Jacobian is zero elsewhere.\\
This indicates a possible usefulness of partitioning each wall $W^i$ into {\em zones} of the form 
$$
\zo^i(\sigma)=\left\{x\in W^i\,|\, \sgn(\phi_j^i-x_j)=\sigma_j\,\forall j \right\}\quad\text{where}\; \sigma\in\{-1,+1\}^{\NN{n}\setminus\{s_i\}}.
$$
And actually, we will show now that for a periodic sequence $\mathcal{C}=\{a^0\cdots a^{\ell-1}\}$ with the two properties mentioned at the beginning (aligned focal points and one exit direction for each $a^i$), and the fact that $\cal C$ is 'full dimensional', there is a single zone of interest on each wall.
\begin{prop}\label{prop-sigmai}
Suppose that along $\cal C$ all variables switch at least once : $\{s_i\,|\, i\in\NN{\ell}\}=\NN{n}$.
Then under the hypotheses above, for each $i\in\{0\cdots \ell-1\}$, there exists a unique $\sigma^i\in\{-1,+1\}^{\NN{n}\setminus\{s_i\}}$ such that
\begin{itemize}
\item $\forall x\in W^i, \; T^{(i)}\circ T^{(i-1)}\circ\cdots T^{(1)}\circ T^{(\ell)} \cdots \circ T^{(i+1)} (x) \in\zo^i(\sigma^i)$ : all orbits eventually enter this zone.
\item $T^{(i)}\left(\zo^i(\sigma^i)\right) \subset \zo^{i+1}(\sigma^{i+1})$ : no orbit escapes these zones.
\end{itemize}
\end{prop}
\begin{proof}
Let $i\in\NN{\ell}$. We are in fact going to define explicitly the sign vector $\sigma^i$. \\
First recall the alignment condition on focal points:
$$
\forall j\ne s_{i+1},\quad \phi_j^i=\phi_j^{i+1}.
$$
Recall also that $\varepsilon_i$ defines whether $W^i$ is crossed increasingly or decreasingly in its normal direction, $s_i$. Then, one can then see from the definitions that:
$$
\varepsilon_i = \sgn(\phi^i_{s_i}-  \theta^{(i)}_{s_i}) =  \sgn(\phi^i_{s_i}-  x_{s_i}) ,\qquad \forall x\in W^{i-1}.
$$
Now,  let $x^0\in W^0$ be an arbitrary point, and then define inductively $x^{i+1}=T^{(i)} x^i$, so that $\{x^i\}_i$ is a trajectory in $\mathcal{C}$.\\
Let us introduce yet another notation: $\Delta_j^i=\sgn(\phi^i_j-x^i_j)$. This quantity depends on $x^0$, at least at first sight. \\
From the alignment condition and the expression of $T^{(i+1)}$ one gets:
$$
 \phi^{i+1}_j -  x_j^{i+1}=  - \alpha^{(i+1)}_j( x^i_{s_{i+1}})\cdot (x^i_j-\phi^i_j),\qquad  j\ne s_{i+1}
$$
And since $\alpha^{(i+1)}_j$ is nonnegative, we have $\forall  j\ne s_{i+1},\;\Delta_j^{i+1}=\Delta_j^i$.\\
Furthermore, the expression of $\varepsilon_i$ given above shows that in direction $s_i$, the precise value of coordinate $x^i_{s_i}$ has no influence. In terms of our new notation: $\Delta_{s_{i}}^{i}=\varepsilon_{i}$ independently of $x^0$, since by construction $x^{i}_{s_{i}}=\theta_{s_{i}}^{(i)}$.\\
Then, one should remark that
\begin{eqnarray*}
\Delta^{i+1}_{s_i} & = & \sgn\left(\phi_{s_i}^{i+1} - (\phi_{s_i}^{i+1} +\alpha^{(i+1)}_{s_i}(x^i_{s_{i+1}})\cdot(x^i_{s_i}-\phi^{i+1}_{s_i}))\right)\\
 & = & \sgn\left(\phi^{i+1}_{s_i} - x^i_{s_i}\right) =  \sgn\left(\phi^{i+1}_{s_i} - \theta^{(i)}_{s_i}\right)
\end{eqnarray*}
since $x^i\in W^i\subset\{x\,|\, x_{s_i}= \theta^{(i)}_{s_i}\}$. Here the important point is that $\Delta^{i+1}_{s_i}$ does not depend on $x^0$, and an easy induction shows that this holds for all subsequent $\Delta^{m}_{s_i}$, when $m\geqslant i+1$.\\
So if all variables switch at least once, as we have supposed, $\Delta^\ell_j$ is independent of the initial condition, i.e. $\sgn(\TT (x^0)-\phi^\ell)$ is fixed. Starting from other walls than $W^0$ does not change the argument.\\

As for invariance, it is a consequence of invariance of the zones $\zo^i(\sigma)$ under the action of the flow in each box, which can be retrieved from the explicit expression of this flow.
\end{proof}

\subsection{Main result}
We can now recapitulate the results of the previous section, and use them to apply theorem \ref{thm-smith}. This will be formulated as a single theorem, but before that we make some remarks.\\
If the condition of proposition \ref{prop-sigmai} is verified (all variables switch), but not the condition of proposition \ref{prop-parallel}, then it is not difficult to see that $\bigcap_{j=0}^\ell W^j$ is a single point, which is furthermore a fixed point of $\TT$. In the statement and proof of the theorem below, this point is denoted $0$.\\
When the condition of proposition \ref{prop-parallel} holds, some directions involve several distinct thresholds, and $\bigcap_{j=0}^\ell W^j = \varnothing$. In this case, $0$ will denote the corner point of $W^0$ (i.e. the boundary point with all its coordinates being threshold values), which also belongs to $\zo^0(\sigma^0)$.\\
Finally, it is clear that the map $\TT$ is differentiable inside $W^0$ and that its differential $D\TT$ can always be extended by continuity to the point we conventionally denote $0$.\\
In words, the theorem below states that given a cycle of regular domains where successive pairs of focal points are aligned and all variables switch, there exists either a unique stable and attracting periodic orbit, or $0$ is the only attractor. The alternative depends on the stability of $0$, and if furthermore two parallel thresholds are crossed, $0$ is not fixed and there is a unique stable periodic orbit.
\begin{theo}\label{thm-main}
 Let ${\cal C}=\{a^0,a^1\cdots a^{\ell-1}\}$ denote a sequence of regular domains which is periodically visited by the flow, and such that each domain $a^i$ has a unique exiting direction $s_i$. Suppose that the focal points of $\cal C$ satisfy Assumption \ref{assum}, i.e. they are aligned. Suppose also that all variables are switching at least once.\\
Consider the first return map $\TT:W^0\to W^0$ defined in (\ref{eq-T}). Let $\lambda =\rho(D\TT(0))$, the spectral radius of $D\TT(0)$.
Then, the following alternative holds:
\begin{itemize} 
\item[i)] if $\lambda\leqslant 1$, then $\forall x\in W^0$, $\TT^n x\to 0$ when $n\to\infty$. 
\item[ii)] if $\lambda>1$ then there exists a unique nonzero fixed point $q=\TT q$. Moreover, $q\in  \intr{\zo^0(\sigma^0)}$ and for every $x\in W^0\setminus\{0\}$, $\TT^n x\to q$ as $n\to\infty$. 
\end{itemize} 

If moreover the condition of proposition \ref{prop-parallel} is satisfied, then the conclusion of $ii)$ holds.
\end{theo}
\begin{proof}
 To prove this statement, we verify that the hypotheses of theorem \ref{thm-smith} are satisfied. First, from proposition \ref{prop-sigmai}, any $x\in W^0$ enters $\zo^0(\sigma^0)$ under the action of $\TT$. Thus, we can consider the restriction of $\TT$ to this zone from now on. We now give the main arguments for a proof that this restriction is monotone and concave, in the sense of conditions $(M)$ and $(C)$.\\
From equation \ref{eq-dTidxj}, it follows that the Jacobian of each local transition map has fixed sign in the zone $\zo^i(\sigma^i)$, with nonzero terms only on the diagonal and on the column $s_i$. From the chain rule, it follows that the Jacobian of $\TT$, denoted $J\TT$, is the product of matrices of this form. Since $s_i$ takes all values in $\NN{n}$ by assumption, $J\TT$ has no zero terms, and from proposition \ref{prop-sigmai} one can easily deduce that these terms have a fixed sign on $\zo^0(\sigma^0)$. Then, by a simple change of coordinate system, it is always possible to ensure that these terms are in fact positive (see \cite{negloop,actabio} for an explicit example of this change of coordinates).\\
Now, the condition $(C)$ can be proved to hold thanks to proposition \ref{prop-D2C} and the explicit form of second order derivatives of local transition maps:
\begin{equation}\label{eq-d2Tidxj}
\frac{\partial^2\T_k^{(i)}}{\partial x_m\partial x_j} (x) =\left\{
\begin{array}{ll}
\displaystyle\frac{\gamma_k}{\gamma_{s_i}} \frac{ \alpha_k^{(i)}(x)} {\phi_{s_i}-x_{s_i}}
& \text{if }\;  k=j,\,m=s_i\quad\text{or } j=s_i,\,m=k\\[3mm]
\displaystyle -\frac{\gamma_k}{\gamma_{s_i}} \left(1+\frac{\gamma_k}{\gamma_{s_i}}\right) \frac{\phi_k^i -
x_k}{(\phi_{s_i}^i - x_{s_i})^2}  \alpha_k^{(i)}(x)  & \text{if }\; 
m=j=s_i \\[3mm]
0 & \text{otherwise}.
\end{array}\right.
\end{equation}
where it appears that these quantities are of fixed sign in the zones $\zo^i(\sigma^i)$. Then we can apply proposition \ref{prop-D2C}, up to the same coordinate change as for the proof of $(M)$. Because transition maps are monotone (up to coordinate change) it is possible to prove that the full return map $\TT$ is also concave, just as the local maps. Here again, we refer to \cite{negloop,actabio} for a detailed justification.\\

The statements $i)$ and $ii)$ are now obtained by strictly applying theorem \ref{thm-smith}, the only missing hypothesis being $\TT p<p$. The latter can be proved using similar arguments as those in the beginning of the proof of proposition \ref{prop-parallel}. Here again, we refer to \cite{negloop,actabio} for a detailed proof.\\
Finally, the last part of the theorem is a direct consequence of proposition \ref{prop-parallel}. Actually in this case, whatever corner point has been chosen as origin, the proposition shows that it is mapped in the interior of $\zo^0(\sigma^0)$, which is identical to the condition $T0\gneq 0$ of theorem \ref{thm-smith}, whence the identical conclusion.
\end{proof}

 This new result bears some resemblance with the previous work on so-called {\it cyclic attractors} \cite{glasspastern1,glasspastern2}. The important improvement is that we do not make any assumption on decay rates, which were uniform in this previous result. On the other hand, we have to suppose the alignment condition on focal points for our result to apply, whereas the cited references do not make any assumption on the precise position of focal points.

\section{Examples}
In this section we study three examples. They are not directly inspired by real biological systems, but serve the purpose of illustrating potential applications of theorem \ref{thm-main}.

\subsection{Two intricate negative loops}\label{sec-2loops}
Our previous result \cite{negloop,actabio} can be seen as a particular case of theorem \ref{thm-main}: for a system consisting of a single negative feedback loop involving all variables, there exists a cycle $\cal C$ satisfying the conditions of this theorem. Moreover in this case, we showed that $i)$ holds with two variables (by showing $\lambda=1$), and $ii)$ holds with three variables or more (by showing $\lambda>1$).\\

In this section we will consider systems with three variables consisting of two negative loops intricate in the following manner:
$$
\SelectTips{eu}{11}\xymatrix @ -.5pc {*+[o][F-] {1}  \ar@<1mm>@*{[|<.5pt>]}@/^.8pc/@{-|}[rr]   && *+[o][F-] {2} \ar@<1mm>@*{[|<.5pt>]}@/^.8pc/[dl] \\
 & *+[o][F-] {3}\ar@<1mm>@*{[|<.5pt>]}@/^.8pc/[ul] \ar@<1mm>@*{[|<.5pt>]}@/^.8pc/@{-|}[ur]&}
$$
First, the graph above corresponds in fact to various different systems. In particular, since gene $3$ acts on both other genes, this may happen in general at two distinct threshold values, whose order needs to be specified. Moreover, the production rate of gene $3$ is a function of $x_2$ of the form $K_3\: {\sf s}^+(x_2)$, and the value $K_3/\gamma_3$ relative to these two thresholds can be chosen in two qualitatively distinct ways: between the two thresholds, or higher than the greatest one (excluding the case where it is lower than the $\min$ of the two thresholds). Also, since gene $2$ is regulated by the two other genes, we must also chose whether this happens as a sum or a product of step functions (a rapid inspection shows that in purely boolean terms, these are the only two cases in accordance with the graph above). In summary, there are thus eight cases to distinguish (two choices of thresholds orders, two choices of production rate $K_3$ and two choices of input function on $x_2$).\\
In all cases, the set of regular domains is represented by ${\cal A}=\{0,1\}^2\times\{0,1,2\}$. Any of the four choices fixes the position of the 12 focal points, and thus a discrete transition structure which may contain a cyclic sequence of the type studied in this paper. According to our inspection, in two cases this leads to a structure with no cycle, in one case to a structure with two cycles involving 2 switching variables only, in four cases to a structure with one or more cycles presenting escaping edges (or walls) and finally in a single case to a structure with a cycle involving all three variables, and no escaping edge. In the following we analyse this last case only: it corresponds to $\theta_3^1$ (resp. $\theta_3^2$) being the threshold of the activation $3\to 1$ (resp. the inhibition $3\dashv 2$), and the regulation of $x_2$ being a product of step functions, with constraints on $K_3$ as below: 
$$\left\{
\begin{array}{l}
\dot x_1(t)  =  K_1\:{\sf s}^+(x_3,\theta_3^1) -\gamma_1\: x_1\\
\dot x_2(t)  =  K_2\: {\sf s}^-(x_1)\:{\sf s}^-(x_3,\theta_3^2) -\gamma_2\: x_2\\
\dot x_3(t)  =  K_3\: {\sf s}^+(x_2)-\gamma_3\: x_3\\[3mm]
\text{with constraints  }\quad K_i>\theta_i\gamma_i,\;\; i=1,2,\quad \theta_3^2\:\gamma_3\;> K_3\;>\theta_3^1\:\gamma_3.
\end{array}
\right.$$

Up to a division by the degradation rates $\gamma_i$ the focal points are given be the following table:
$$\begin{array}{lc|c|c|c|c|c|c|c|c|c|c|c}
{\cal C}:&000 & 001 & 002& 010& 011& 012& 100& 101& 102& 110& 111& 112 \\
\hline 
&0    &K_1&K_1&0    &K_1&K_1&0&K_1&K_1&0     &K_1&K_1\\
&K_2&K_2&0    &K_2&K_2&0    &0&0      &0    &0     &0    &0    \\
&0    &0    &0    &K_3&K_3&K_2&0&0      &0    &K_3&K_3&K_3
\end{array}$$
Then, the transition structure mentioned above can be depicted as follows:
$$
\xymatrix @ -.5cm {&110  \ar@*{[|<.5pt>]}[rr]\ar@*{[|<.5pt>]}'[d][dd]\ar@*{[|<.5pt>]}[dl] && 111\ar@*{[|<1.8pt>]}[dl]  && 112\ar@*{[|<.5pt>]}[dl] \ar@*{[|<.5pt>]}[ll]  \\
100  \ar@*{[|<1.8pt>]}[dd] && 101\ar@*{[|<1.8pt>]}[ll]  && 102\ar@*{[|<.5pt>]}[ll]  &\\
&010  \ar@*{[|<1.8pt>]}'[r][rr] && 011\ar@*{[|<1.8pt>]}'[u][uu]  && 012 \ar@*{[|<.5pt>]}'[l][ll] \ar@*{[|<.5pt>]}[uu]  \ar@*{[|<.5pt>]}[dl] \\
000  \ar@*{[|<1.8pt>]}[ur] && 001 \ar@*{[|<.5pt>]}[ll]\ar@*{[|<.5pt>]}[uu]\ar@*{[|<.5pt>]}[ur]   && 002\ar@*{[|<.5pt>]}[uu]\ar@*{[|<.5pt>]}[ll]& 
  }
$$
where nodes represent regular domains, and arrows represent the transitions imposed by the flow (or the position of focal points). We have underlined the cycle $\cal C$ using bold arrows. It is easily seen that any of the 6 regular domains not belonging to $\cal C$ contain only initial conditions which enter $\cal C$ in finite time. In fact, we can remark that the 4 domains with $a_3=2$, i.e. the half-space $x_3>\theta_3^2$ is repelling for the flow. Hence, we can restrict the study of this system to the remaining half-space, which can be done be fixing ${\sf s}^-(x_3,\theta_3^2)=1$ in the equations. Then, we immediately see that the obtained system is a negative feedback loop with three variables. It is easy to check that the focal points in the cycle satisfy assumption \ref{assum}, for instance from the table above. Thus, we can apply theorem \ref{thm-main}, and moreover, we also know from \cite{negloop,actabio} that the origin is unstable (i.e. $\lambda>1$), hence the conclusion $ii)$ holds. In brief: the cycle $\cal C$ contains a unique stable periodic orbit, which attracts all initial conditions in the 12 regular domains, as illustrated on figure \ref{fig-2loops}.
\begin{center}
\begin{figure}\begin{center}
\includegraphics[scale=0.4]{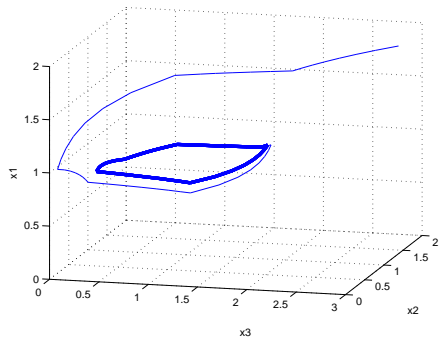} 
\caption{\label{fig-2loops} The limit cycle resulting (represented with a thicker line than the transient regime) from the example in section \ref{sec-2loops}. Parameter values are : $\theta_1=\theta_2=\theta_3^1=1$ and $\theta_3^2=2$. $(\gamma_1,\gamma_2,\gamma_3)=(1,3,6)$. $K_1=1.3$, $K_2=5.1$, $K_3=10.8$. Initial condition: $x(0)=(1.8,1.9,2.8)$ is chosen outside the region where the cycle lies, to indicate that the basin of attraction of this cycle is the whole space.}
\end{center}\end{figure}
\end{center}

\subsection{More complex interaction graph}\label{sec-ex-curvy}
In this example, we chose a more complex interaction structure, involving multiple loops of different signs. Namely, we apply our result to the following system:\\

\begin{tabular}{cc}
$\SelectTips{eu}{11}\xymatrix @ -.5pc {*+[o][F-] {1}  \ar@*{[|<.5pt>]}[rr] \ar@*{[|<.5pt>]}@/_.8pc/[dr]  && *+[o][F-] {2}  \ar@*{[|<.5pt>]}@/_.8pc/[ll]\ar@*{[|<.5pt>]}[dl]\ar@*{[|<.5pt>]}@(r,u) \\
 & *+[o][F-] {3}\ar@*{[|<.5pt>]}@/_.8pc/[ur] \ar@*{[|<.5pt>]}[ul] \ar@*{[|<.5pt>]}@(r,d)&}
\qquad$ & 
$\left\{
\begin{array}{l}
\dot x_1(t)  =  [K_1{\sf s}^-(x_2)+ K_1'{\sf s}^+(x_2)]\:{\sf s}^-(x_3) + K_1''{\sf s}^-(x_2)\:{\sf s}^+(x_3) -\gamma_1\: x_1\\
\dot x_2(t)  =  K_2\left[{\sf s}^-(x_1)\:{\sf s}^-(x_3) + {\sf s}^+(x_1)\:{\sf s}^+(x_2)\right]-\gamma_2\: x_2\\
\dot x_3(t)  =  K_3\left[{\sf s}^-(x_1)\:{\sf s}^+(x_3) + {\sf s}^+(x_1)\:{\sf s}^+(x_2)\right]-\gamma_3\: x_3\\[3mm]
\text{with constraints  }\quad K_i>\theta_i\gamma_i,\;\; i=2,3.\\
\phantom{with constraints }\quad K_1<\theta_1\gamma_1,\;\;K_1'>\theta_1\gamma_1,\;\;K_1''>\theta_1\gamma_1.
\end{array}
\right.$
\end{tabular}\ \\[6mm]
whose interaction graph is depicted on the left: its arrows correspond to the interactions appearing in the production term, so that there is no explicit self loop on the node $1$. No distinction between activation and repression is made in this graph, to avoid multiple arrows. Actually, the action of $2$ on $1$ can be of both types depending on $x_3$ and similarly for several other arrows.\\ 
The step functions act with a single threshold $\theta_i$ for each variable $x_i$, not shown in the equations above. The eight parameters $K_i$ and $\gamma_i$ may take any value satisfying the consistency constraints.\\
Since there is a single threshold per variable, the regular domains of this system can be represented by $\A=\{0,1\}^3$. Up to a division by the degradation rates $\gamma_i$ the focal points are given by the following table:
$$\begin{array}{lc|c|c|c|c|c|c|c}
{\cal C}:&000 & 010 & 110 & 111 & 011 & 001 & 101 & 100 \\
\hline 
&K_1 &K_1'&K_1'&0    &0     & K_1''&K_1''&K_1 \\
&K_2 &K_2 &K_2  &K_2&0    & 0       &0      &0\\
&0     & 0    &K_3  &K_3&K_3 & K_3  &0      &0
\end{array}$$
In a more geometrical way, this cycle can be depicted as:
$$
\xymatrix @ -.5cm {&011  \ar@*{[|<1.8pt>]}[dl] && 111\ar@*{[|<1.8pt>]}[ll]  \\
001  \ar@*{[|<1.8pt>]}[rr] && 101\ar@*{[|<1.8pt>]}[dd] \ar@{.}[ur] & \\
&010  \ar@*{[|<1.8pt>]}'[r][rr] \ar@{.}[uu] && 110\ar@*{[|<1.8pt>]}[uu]   \\
000  \ar@*{[|<1.8pt>]}[ur] \ar@{.}[uu] && 100 \ar@*{[|<1.8pt>]}[ll] \ar@{.}[ur]  &}
$$
The transitions between regular domains in ${\cal A}=\{0,1\}^3$. The dotted lines represents a white walls: since no trajectory starting outside these walls can reach them, we can ignore them without difficulty.\\

In the table above, we have disposed the domains in the order followed by any trajectory under the given parameter constraints. In particular, there is a cycle $\cal C$ involving all the regular domains of this system. Finally, it is easily seen in this table that two consecutive focal points only differ in the switching direction: they are aligned in the sense require by theorem \ref{thm-main}. Moreover, the three directions are switching. Hence, our theorem applies and this system admits a unique attractor, which may be either the 'origin' $(\theta_1,\theta_2,\theta_3)$ or a stable limit cycle. We have not been able to prove that the origin is unstable, but all numerical simulations we have performed with various parameter values have led to a limit cycle, as illustrated in figure \ref{fig-curvy}.\\
\begin{center}
\begin{figure}[hbtp]\begin{center}
\includegraphics[scale=0.4]{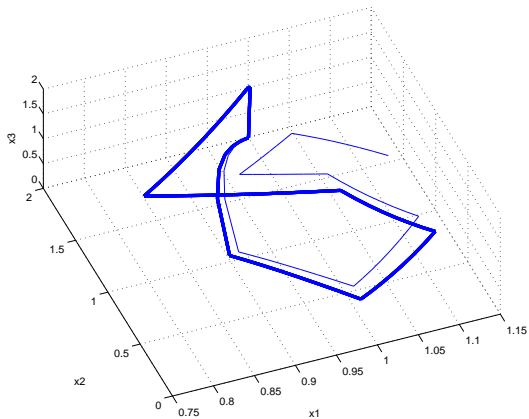} 
\caption{\label{fig-curvy} An example of numerical simulation of the example from section \ref{sec-ex-curvy}. A limit cycle is observed, and represented with a thick line in the figure. Parameter values are : $\theta_i=1$ for all $i$. $(\gamma_1,\gamma_2,\gamma_3)=(1,3,6)$. $K_1=0.4$, $K_1'=1.02$, $K_1''=2$, $K_2=6$, $K_3=12$. Initial condition: $x(0)=(1.1,1.1,1.1)$.}
\end{center}\end{figure}
\end{center}
Despite its complicated interaction graph, this system is such that on any wall, each variable modifies the value of at most one focal point coordinate. This property is the origin of the alignment of focal point, and might be used to study more examples with a complex global interaction graph, which simplifies locally.\\
Our last comment on this example concerns assumption \ref{assum}. Actually, this alignment assumption could be thought to imply that focal points are vertices of a rectangular parallelepiped. However, a rapid inspection of the table above will show to the reader that this is not the case in our example, and that it is difficult to draw global conclusions from the local condition of focal point alignment.\\

\subsection{Multiple threshold values}
We provide now a simple two dimensional example involving two distinct thresholds in one direction. Namely, $x_1$ can switch at two thresholds $\theta_1^1<\theta_1^2$, and $x_2$ at a single threshold denoted $\theta_2$ (and omitted in step functions involving $x_2$). This system writes\\

\begin{tabular}{c}
$\SelectTips{eu}{11}\xymatrix {*+[o][F-] {1}  \ar@*{[|<.5pt>]}@/_.8pc/[r]\ar@*{[|<.5pt>]}@(lu,ld)  & *+[o][F-] {2}  \ar@*{[|<.5pt>]}@/_.8pc/[l]\ar@*{[|<.5pt>]}@(ru,rd)}\quad$ \\[4mm]
$\left\{
\begin{array}{l}
\dot x_1(t)  =  K_1\:[{\sf s}^-(x_1,\theta_1^1)\:{\sf s}^-(x_2) +{\sf s}^+(x_1,\theta_1^2) \:{\sf s}^+(x_2)] + K_1'\:{\sf s}^+(x_1,\theta_1^1)\:{\sf s}^-(x_2) -\gamma_1\: x_1\\[2mm]
\dot x_2(t)  =  K_2\:[ {\sf s}^+(x_1,\theta_1^1)\:{\sf s}^+(x_2) + {\sf s}^+(x_1,\theta_1^2)\:{\sf s}^-(x_2)] + K_2' -\gamma_2\: x_2\\[3mm]
\text{with constraints  }\quad K_1>\theta_1^1\gamma_1,\;\;K_1'>\theta_1^2\gamma_1,\;\; K_2'<\theta_2\gamma_2,\;\;K_2+K_2'>\theta_2\gamma_2.\\
\end{array}
\right.$
\end{tabular}\ \\[6mm]
The aim of this model is to illustrate the last statement of theorem \ref{thm-main}. First, let us show the production rate values (or focal point coordinates multiplied by decay rates) in a table, as in the previous example:
$$\begin{array}{lc|c|c|c|c|c}
{\cal C}:&00 & 10 & 20 & 21 & 11 & 01 \\
\hline 
&K_1&K_1'&K_1'&K_1&0     & 0 \\
&K_2'&K_2'&K_2+K_2' &K_2+K_2'&K_2+K_2'&K_2'
\end{array}$$
Here again, a cycle $\cal C$ involving all regular domains exists for any parameter set satisfying the specified constraints. Moreover, any pair of successive focal points only differ in the switching direction, i.e. assumption \ref{assum} is verified. Hence, we may apply theorem \ref{thm-main}, and since $\theta_1^1$ and $\theta_1^2$ are both crossed in $\cal C$, we conclude that there exists a unique stable periodic orbit attracting all initial conditions. This fact is illustrated on figure \ref{fig-2cyc}.\\
\begin{center}
\begin{figure}\begin{center}
\begin{tabular}{cc}$
\xymatrix {01  \ar@*{[|<.5pt>]}[d] & 11\ar@*{[|<.5pt>]}[l] \ar@{.}@*{[|<.5pt>]}[d] & 21\ar@*{[|<.5pt>]}[l]\\
00  \ar@*{[|<.5pt>]}[r] & 10\ar@*{[|<.5pt>]}[r] & 20\ar@*{[|<.5pt>]}[u]  }$  &
\includegraphics[scale=0.4]{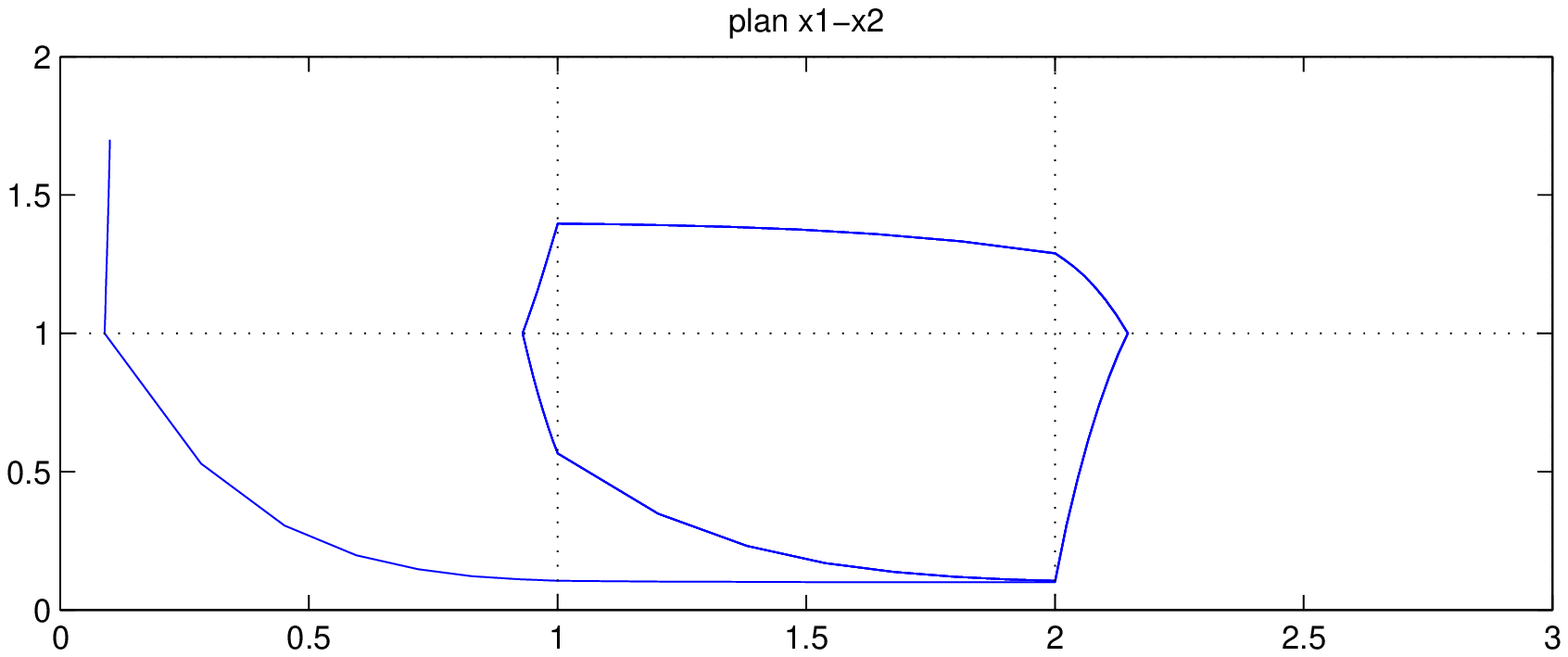}
\end{tabular}
\caption{\label{fig-2cyc} On the left: the transitions between regular domains in ${\cal A}=\{0,1,2\}\times\{0,1\}$. The dotted line represents a white wall: since no trajectory starting outside this wall can reach it, we can ignore this wall without difficulty. On the right: an example of numerical simulation of the example from section \ref{sec-ex-curvy}. A limit cycle is observed. Parameter values are : $\theta_1^1=\theta_2=1$, and $\theta_1^2=2$ (dotted lines represent these thresholds). $(\gamma_1,\gamma_2)=(1,5)$. $K_1=1.5$, $K_1'=2.7$, $K_2'=0.5$, $K_2=6.5$.}
\end{center}\end{figure}
\end{center}
We would like to stress the fact that this example has been chosen for its geometrical simplicity, which gives an easy intuition of our result. However, more complex cases could also be treated, and in particular we recall that the parallel thresholds may not be crossed successively. Let us illustrate this on a discrete transition structure, without entering into the detail of underlying differential equations. For example, if we now that all pairs of successive focal points are aligned in the bold cycle below:
$$
\xymatrix @ -.5cm {&110  \ar@*{[|<1.8pt>]}'[d][dd]\ar@{.}@*{[|<.5pt>]}[rr] && 111\ar@*{[|<1.8pt>]}[rr]  && 112\ar@*{[|<1.8pt>]}[dl]  \\
100  \ar@*{[|<1.8pt>]}[ur] && 101\ar@{.}@*{[|<.5pt>]}[ur]\ar@{.}@*{[|<.5pt>]}[rr]\ar@*{[|<1.8pt>]}[ll]  && 102\ar@*{[|<1.8pt>]}[dd]  &\\
&010  \ar@*{[|<1.8pt>]}'[r][rr] && 011\ar@*{[|<1.8pt>]}'[u][uu]  && 012 \ar@*{[|<.5pt>]}'[l][ll] \ar@*{[|<.5pt>]}[uu]  \ar@*{[|<.5pt>]}[dl] \\
000  \ar@*{[|<.5pt>]}[rr]\ar@*{[|<.5pt>]}[ur]\ar@*{[|<.5pt>]}[uu] && 001 \ar@*{[|<1.8pt>]}[uu]\ar@{.}@*{[|<.5pt>]}[ur]   && 002 \ar@*{[|<1.8pt>]}[ll]& 
  }
$$
then, we can conclude that there exists a unique stable periodic solution of the corresponding piecewise-linear differential equations, because two thresholds are crossed in the left-right direction. The above example may seem a little artificial, but it serves as an illustration of a property that holds in higher dimensional systems, whose projections in 3 dimension might be at least as complicated as the diagram above.

\section{Conclusion}\label{sec-ccl}
We have presented and proved in this paper a theorem about periodic solutions of piecewise linear models of gene regulatory networks. This theorem relates discrete transitions between regular domains of these systems and their actual solutions. It does so under hypotheses of a local nature (alignment of pairs of successive focal points), hence allowing applications to a large variety of examples, some of which have been presented in the last section. The alignment condition is related to the interaction structure of the system only locally, allowing complicated global interaction graphs to be handled within the present framework.\\
A possible follow-up of this work would concern the control of gene networks, a topic arisen recently as a tool for synthetic biology. Actually, we have shown in a previous paper \cite{automatica} that if production and degradation rates can be modified by an experimentalist (a fact modeled using input variables), then the problem of forcing a discrete transition structure in piecewise affine models could be expressed as a simple linear programming (LP) problem. We applied this to the control of steady states, but it could also be applied to the control of stable periodic orbits. Actually, if it is possible to control the discrete transition structure so that it presents a cycle, and if moreover we are able to impose that successive focal points are aligned (which amounts to adding constraints to the LP problem), then the result of this paper shows that the controlled system will have only two possible attractors: the origin or a stable limit cycle. This could be a powerful tool for regulating a system between two functioning modes.\\

\appendix
\noindent{\bf \Large Appendix}
\section{Proof of proposition \ref{prop-parallel}}\label{anx-proofs}
\begin{proof}
Let us deal first with the more intuitive case when the two thresholds $\theta^{(0)}_{s_0}$, $\theta^{(k)}_{s_k}$ are crossed successively, i.e. there are two consecutive parallel walls. In other terms, one assumes first that $k=1$. Then, for any $x\in W^0$, $x_{s_1}=x_{s_0}=\theta^{(0)}_{s_0}$, and eq.~(\ref{eq-loctransmap}) gives
$$
{T}^{(1)}_j(x) =  \phi^1_j + \alpha^{(1)}_j\!\left(\theta_{s_0}^{(0)}\right)\:(x_j - \phi^1_j) ,\quad j\ne s_1.
$$
It is easily checked that either $\theta^{(0)}_{s_0}<\theta^{(1)}_{s_1}<\phi^1_{s_1}$ or $\theta^{(0)}_{s_0}>\theta^{(1)}_{s_1}>\phi^1_{s_1}$, since trajectories leave $a^0$ to enter $a^1$. It follows that the scalar $ \alpha^{(1)}_j\! \left(\theta_{s_0}^{(0)}\right)\in(0,1)$. Since it does not depend on $x$, we abbreviate it into $\alpha^{(1)}_j$. Now, remember that $s_1$ is by assumption the only exit direction from $a^1$. This implies that for all $j\ne s_1$, denoting $\theta_j^\pm$ the thresholds bounding $W^0$ and $W^1$ in direction $j$, one has: $\theta_j^+>\phi_j^1>\theta_j^-$, or equivalently 
$$
\theta_j^+ - \phi_j^1>0>\theta_j^-- \phi_j^1.
$$
Multiplying by $\alpha^{(1)}$, this implies 
$$
\alpha^{(1)}\cdot [\theta_j^- - \phi_j^1,\theta_j^+ - \phi_j^1]\subset (\theta_j^- - \phi_j^1,\theta_j^+ - \phi_j^1).
$$
Since the left-hand side above is the image of $ [\theta_j^-,\theta_j^+]$ by the translated map $x\mapsto{T}^{(1)}_j(x) - \phi^1_j$, one deduces the expected inclusion: $T^{(1)} \left(W^0\right) \subset \intr{W^1}$.\\

Now suppose $k>1$.\\
Since the walls are closed, connected sets, and all maps $T^{(i)}$ are continuous, to show that some wall is mapped in the interior of another wall, it is sufficient to show that none of its point is mapped on the boundary of the target wall. To achieve this, let us first prove that the following equality holds for any $i\in\NN{\ell}$:
\begin{equation}\label{eq-induc}
 T^{(i)}\circ T^{(i-1)}\circ\cdots\circ T^{(1)} \left(W^0\right) \cap \partial W^i = \bigcap_{j=0}^i W^j.
\end{equation}
The proof is by induction. To initialise this induction, let us consider how the boundary of the target wall of the first transition map $T^{(1)}$ is intersected. That is, we describe 
$T^{(1)}\left( W^0 \right)\cap \partial W^1$. 
Once again it will be useful to consider the translated map:
\begin{equation}\label{eq-T1}
T_j^{(1)}(x) - \phi^1_j = \alpha^{(1)}_j\left( x_{s_1}\right)(x_j - \phi^1_j) ,\quad j\ne s_1.
\end{equation}
Let us denote $\theta_j^- < \theta_j^+$ the bounding thresholds of $W^1$ in direction $j$. Note that these are also the bounding thresholds of $W^0$, for $j\ne s_0$. We also denote $\theta_{s_1}^- < \theta_{s_1}^+$ the bounding thresholds of $W^0$ in direction $s_1$, so that $W^0=\prod_{j<s_0}[\theta_j^- , \theta_j^+] \times \{\theta_{s_0}^{(0)}\} \times  \prod_{j>s_0}[\theta_j^- , \theta_j^+]$ and $W^1=\prod_{j<s_1}[\theta_j^- ,\theta_j^+] \times \{\theta_{s_1}^{(1)}\} \times  \prod_{j>s_1}[\theta_j^- , \theta_j^+]$.\\
Now, an image point encounters the boundary of $W^1$ if and only if $T_j^{(1)}(x)=\theta_j^\pm$ for some $j\ne s_1$. Considering~(\ref{eq-T1}), this is equivalent to:
\begin{equation}\label{eq-transT1}
T_j^{(1)}(x) - \phi^1_j = \theta_j^\pm - \phi^1_j = \alpha^{(1)}_j\left( x_{s_1}\right)\:(x_j - \phi^1_j)
\end{equation}
for one of the two values of $\theta_j^\pm$. Now, observe that:
\begin{itemize}
\item since $s_1$ is the only exiting direction, one has $\phi_j\in (\theta_j^- ,\theta_j^+)$ for all $j\ne s_1$.
\item $\alpha^{(1)}_j\left( x_{s_1}\right)\in (0,1]$ for $x_{s_1}\in [\theta_{s_1}^-, \theta_{s_1}^+]$, it is a monotone function of $x_{s_1}$, and takes the value $1$ only for $x_{s_1}=\theta_{s_1}^{(1)}$.
\end{itemize}
The first observation above implies that $\theta_j^- - \phi^1_j<0< \theta_j^+ - \phi^1_j$ for $j\ne s_1$, and for any $\alpha\in(0,1)$ and $x_j\in [\theta_j^- , \theta_j^+]$ this in turn gives:
$$
\theta_j^- - \phi^1_j < \alpha\cdot(\theta_j^- - \phi^1_j) \leqslant \alpha\cdot (x_j-\phi^1_j)\leqslant \alpha\cdot (\theta_j^+ - \phi^1_j) < \theta_j^+ - \phi^1_j
$$
Hence, one sees that~(\ref{eq-transT1}) may only be satisfied if $\alpha^{(1)}_j\left( x_{s_1}\right)=1$, which from the second observation above occurs exactly for $x_{s_1}=\theta_{s_1}^{(1)}$. But this defines the hyperplane bearing the target wall $W^1$. In other words, the preimage of $\partial W^1$ by $T^{(1)}$ is
\begin{equation}\label{eq-TWdW}
W^0\cap W^1= \prod_{j=1}^n [\theta_j^- , \theta_j^+]\cap \left\{x\,|\, x_{s_0}=\theta_{s_0}^{(0)},\; x_{s_1}=\theta_{s_1}^{(1)}\right\}.
\end{equation}
Moreover, from $\alpha^{(1)}_j\left( \theta_{s_1}^{(1)}\right)=1$, the restriction of $T^{(1)}$ to $W^0\cap W^1$ is the identity. Hence, the set above is exactly $T^{(1)}\left( W^0 \right)\cap \partial W^1$.\\
Remark that the case $k=1$ treated previously could have been deduced from this property, since in this case one has parallel walls, and thus $W^0\cap W^1 = \varnothing$.\\
Under the assumption that we deal with boxes with a single outgoing direction, the argument above generalises to any local transition map as:
\begin{equation}\label{eq-TidWi}
T^{(i+1)}\left( W^{i} \right)\cap \partial W^{i+1} = W^{i}\cap W^{i+1}.
\end{equation}
 Now another useful observation is that
 \begin{equation}\label{eq-WiWi-1}
 W^{i}\cap W^{i+1} = \partial W^{i}\cap \partial W^{i+1}
 \end{equation}
 as may be seen from the explicit descriptions of each of these sets.\\

Now, suppose that the induction statement (\ref{eq-induc}) holds for some $i$. Then, we have
$$.
T^{(i+1)}\circ T^{(i)}\circ\cdots\circ T^{(1)} \left(W^0\right) \cap \partial W^{i+1} \subset W^i\cap W^{i+1}
$$
by (\ref{eq-TidWi}) and the inclusion $T^{(i)}\circ\cdots\circ T^{(1)} \left(W^0\right)\subset W^i$. Then, because $T^{(i+1)}$ acts as the identity on $W^i\cap W^{i+1}$, the only points of $T^{(i)}\circ\cdots\circ T^{(1)} \left(W^0\right)$ whose image by $T^{(i+1)}$ intersects $\partial W^{i+1}$ must also lie in $W^i\cap W^{i+1}$. But this intersection is a subset of $\partial W^{i}$ by (\ref{eq-WiWi-1}). Hence, these points lie in fact in $T^{(i)}\circ T^{(i-1)}\circ\cdots\circ T^{(1)} \left(W^0\right) \cap \partial W^i$, which by the induction hypothesis equals $\bigcap_{j=0}^i W^j$. Using again the fact that $T^{(i+1)}$ acts as the identity on this set, the expected 
$ T^{(i+1)}\circ T^{(i)}\circ\cdots\circ T^{(1)} \left(W^0\right) \cap \partial W^{i+1}=\bigcap_{j=0}^{i+1} W^j$ follows.\\

Now, it is not difficult to see that~(\ref{eq-TWdW}) can be generalised to the above intersection as follows:
$$
\bigcap_{i=0}^k W^i = \prod_{j=1}^n [\theta_j^- , \theta_j^+]\cap \left\{x\,|\, x_{s_i}=\theta_{s_i}^{(i)},\; i=0\cdots k \right\}
$$
and the main assumption of this proposition implies that this set is empty.
\end{proof}


\end{document}